\documentclass[11pt]{amsart}

\usepackage{amsxtra,amssymb,amsthm,amsmath,url,eucal}
\usepackage[latin1]{inputenc}
\renewcommand{\leq}{\leqslant}
\renewcommand{\geq}{\geqslant}

\newcounter{point}

\newcommand{\Cc}{\mathbf{C}}

\newcommand{\Aa}{\mathbf{A}}

\newcommand{\Zz}{\mathbf{Z}}

\newcommand{\Rr}{\mathbf{R}}

\newcommand{\Hh}{\mathbf{H}}
\newcommand{\Qq}{\mathbf{Q}}

\newcommand{\mods}[1]{\,(\mathrm{mod}\,{#1})}

\newcommand{\T}[1]{{}^t{{#1}}}

\newcommand{\ra}{\rightarrow}
\newcommand{\lra}{\longrightarrow}

\DeclareMathOperator{\rank}{rank}

\DeclareMathOperator{\Imag}{Im}
\DeclareMathOperator{\Reel}{Re}

\DeclareMathOperator{\Tr}{Tr}

\DeclareMathOperator{\Aut}{Aut}

\DeclareMathOperator{\Sp}{Sp}
\DeclareMathOperator{\GL}{GL}
\DeclareMathOperator{\SL}{SL}

\DeclareMathSymbol{\gena}{\mathord}{letters}{"3C}
\DeclareMathSymbol{\genb}{\mathord}{letters}{"3E}

\theoremstyle{plain}
\newtheorem{theorem}{Theorem}

\newtheorem{lemma}[theorem]{Lemma}
\newtheorem{corollary}[theorem]{Corollary}

\newtheorem{proposition}[theorem]{Proposition}

\theoremstyle{remark}
\newtheorem{remark}[theorem]{Remark}

\theoremstyle{definition}

\begin{document}

\title{A note on Fourier coefficients of Poincar\'e series}

\author{Emmanuel Kowalski}
\address{ETH Z\"urich -- D-MATH\\
  R\"amistrasse 101\\
  8092 Z\"urich\\
  Switzerland} \email{kowalski@math.ethz.ch}
\author{Abhishek Saha}
\address{ETH Z\"urich -- D-MATH\\
  R\"amistrasse 101\\
  8092 Z\"urich\\
  Switzerland} \email{abhishek.saha@math.ethz.ch}
\author{Jacob Tsimerman}
\address{Princeton University \\ Fine Hall, Princeton
  NJ 08540, USA}
\email{jtsimerm@math.princeton.edu}



\begin{abstract}
  We give a short and ``soft'' proof of the asymptotic orthogonality
  of Fourier coefficients of Poincar\'e series for classical modular
  forms as well as for Siegel cusp forms, in a qualitative form.
\end{abstract}

\maketitle

\section{Introduction}

The Petersson formula (see, e.g.,~\cite[Ch. 14]{ant}) is one of the
most basic tools in the analytic theory of modular forms on congruence
subgroups of $\SL(2,\Zz)$. One of its simplest consequences, which
explains its usefulness, is that it provides the asymptotic
orthogonality of distinct Fourier coefficients for an orthonormal
basis in a space of cusp forms, when the analytic conductor is large
(e.g., when the weight or the level is large).  From the proof of the
Petersson formula, we see that this orthonormality is equivalent (on a
qualitative level) to the assertion that the $n$-th Fourier
coefficient of the $m$-th Poincar\'e series is essentially the
Kronecker symbol $\delta(m,n)$.
\par
In this note, we provide a direct ``soft'' proof of this fact in the
more general context of Siegel modular forms when the main parameter
is the weight $k$.  Although this is not sufficient to derive the
strongest applications (e.g., to averages of $L$-functions in the
critical strip), it provides at least a good motivation for the more
quantitative orthogonality relations required for those. And, as we
show in our paper~\cite{kst} concerning the local spectral
equidistribution of Satake parameters for certain families of Siegel
modular forms of genus $g=2$, the ``soft'' proof suffices to derive
some basic consequences, such as the analogue of ``strong
approximation'' for cuspidal automorphic representations, and the
determination of the conjectural ``symmetry type'' of the family. See
Corollary~\ref{cor-simple} for a simple example of this when $g=1$.

\par
\medskip
\par
\textbf{Acknowledgements.} Thanks to M. Burger for helpful remarks
concerning the geometry of the Siegel fundamental domain.

\section{Classical modular forms}

In this section, we explain the idea of our proof for classical
modular forms; we hope this will be useful as a comparison point in
the next section, especially for readers unfamiliar with Siegel
modular forms. Let $k\geq 2$ be an even integer, $m\geq 1$ an
integer. The $m$-th Poincar\'e series of weight $k$ is defined by
$$
P_{m,k}(z)=\sum_{\gamma\in \Gamma_{\infty}\backslash \Gamma}{
(cz+d)^{-k}e(m\gamma\cdot z)
},
$$
where $\Gamma=SL(2,\Zz)$, acting on the Poincar\'e upper half-plane
$\Hh$, 
$$
\Gamma_{\infty} =
\Bigl\{\pm\begin{pmatrix}1&n\\0&1\end{pmatrix}: n\in \Zz \Bigr\}
$$
is the stabilizer of the cusp at infinity, and we write
$$
\gamma = \begin{pmatrix}a&b\\c&d \end{pmatrix},\quad\quad
(a,b,c,d)\in\Zz^4.
$$
\par
It is well known that for $k\geq 4$, $m\geq 1$, this series converges
absolutely and uniformly on compact sets, and that it defines a cusp
form of weight $k$ for $\Gamma=\SL(2,\Zz)$. We denote by $p_{m,k}(n)$,
$n\geq 1$, the Fourier coefficients of this Poincar\'e series, so that
$$
P_{m,k}(z)=\sum_{n\geq 1}{p_{m,k}(n)e(nz)}
$$
for all $z\in \Hh$.

\begin{proposition}[Asymptotic orthogonality of Fourier coefficients
  of Poincar\'e series]\label{pr-k}
  With notation as above, for fixed $m\geq 1$, $n\geq 1$, we
  have
$$
\lim_{k\ra +\infty}{p_{m,k}(n)}=\delta(m,n).
$$
\end{proposition}

\begin{proof}
The idea is to use the definition of Fourier coefficients as
$$
p_{m,k}(n)=\int_{U}{P_{m,k}(z)e(-nz)dz}
$$
where $U$ is a suitable horizontal interval of length $1$ in $\Hh$,
and $dz$ is the Lebesgue measure on such an interval; we then let
$k\ra +\infty$ under the integral sign, using the definition of the
Poincar\'e series to understand that limit.
\par
We select
$$
U=\{x+iy_0 \,\mid\, |x|\leq 1/2\}
$$
for some fixed $y_0>1$. The Lebesgue measure is then of course $dx$. 
\par
Consider a term 
$$
(cz+d)^{-k}e(m\gamma \cdot z)
$$
in the Poincar\'e series as $k\ra +\infty$. We have
$$
\Bigl|(cz+d)^{-k}e(m\gamma \cdot z)
\Bigr|\leq |cz+d|^{-k}
$$
for all $z \in \Hh$ and $\gamma\in \SL(2,\Zz)$, since $m\geq 0$ and
$\gamma\cdot z\in \Hh$. But for $z\in U$, we find
\begin{equation}\label{eq-j}
|cz+d|^2=(cx+d)^2+c^2y_0^2\geq c^2y_0^2.
\end{equation}
\par
If $c\not=0$, since $c$ is an integer, the choice of $y_0>1$ leads to
$c^2y_0^2>1$, and hence
$$
\Bigl|(cz+d)^{-k}e(m\gamma \cdot z)
\Bigr|\ra 0
$$
as $k\ra +\infty$, uniformly for $z\in U$ and $\gamma\in \Gamma$ with
$c\not=0$. On the other hand, if $c=0$, we have $\gamma\in
\Gamma_{\infty}$; this means this corresponds to a single term which
we take to be $\gamma=\mathrm{Id}$, and we then have
$$
(cz+d)^{-k}e(m\gamma \cdot z)=e(mz)
$$
for all $k$ and $z\in U$. 
\par
Moreover, all this shows also that
$$
\Bigl|(cz+d)^{-k}e(m\gamma \cdot z)
\Bigr|\leq |cz+d|^{-4}
$$
for $k\geq 4$ and $\gamma\in \Gamma_{\infty}\backslash\Gamma$. Since
the right-hand converges absolutely and uniformly on compact sets, we
derive by dominated convergence that
$$
P_{m,k}(z)\ra e(mz)
$$
for all $z\in U$. The above inequality gives further
$$
|P_{m,k}(z)|\leq \sum_{\gamma\in\Gamma_{\infty}\backslash
  \Gamma}{|cz+d|^{-4}}
$$
for $k\geq 4$ and $z\in U$. Since $U$ is compact, we can integrate by
dominated convergence again to obtain
$$
\int_{U}{P_{m,k}(z)e(-nz)dz}\longrightarrow \int_{U}{e((m-n)z)dz}
=\delta(m,n)
$$
as $k\ra +\infty$.
\end{proof}

It turns out that the same basic technique works for the other most
important parameter of cusp forms, the level. For $q\geq 1$ and $m\geq
1$ integers, let now
$$
P_{m,q}(z)=\sum_{\gamma\in \Gamma_{\infty}\backslash \Gamma_0(q)}{
(cz+d)^{-k}e(m\gamma\cdot z)
}
$$
be the $m$-th Poincar\'e series of weight $k$ for the Hecke group
$\Gamma_0(q)$, and let $p_{m,q}(n)$ denote its Fourier coefficients.

\begin{proposition}[Orthogonality with respect to the level]\label{pr-q}
  With notation as above, for $k\geq 4$ fixed, for any fixed $m$ and
  $n$, we have
$$
\lim_{q\ra +\infty}{p_{m,q}(n)}=\delta(m,n).
$$
\end{proposition}

\begin{proof}
We start with the integral formula
$$
p_{m,q}(n)=\int_{U}{P_{q,m}(z)e(-nz)dz}
$$
as before. To proceed, we observe that $\Gamma_{\infty}\backslash
\Gamma_0(q)$ is a subset of $\Gamma_{\infty}\backslash \Gamma$, and
hence we can write
$$
P_{m,q}(z)=\sum_{\gamma\in \Gamma_{\infty}\backslash
  \Gamma}{\Delta_q(\gamma) (cz+d)^{-k}e(m\gamma\cdot z)
},
$$
where
$$
\Delta_q\Bigl(\begin{pmatrix}a&b\\c&d\end{pmatrix}
\Bigr)=
\begin{cases}
1&\text{ if } c\equiv 0\mods{q},\\
0&\text{ otherwise.}
\end{cases}
$$
\par
We let $q\ra +\infty$ in each term of this series. Clearly, we have
$\Delta_q(\gamma)=0$ for all $q>c$, unless if $c=0$, in which case
$\Delta_q(\gamma)=1$. Thus
$$
\Delta_q(\gamma) (cz+d)^{-k}e(m\gamma\cdot z)\ra 0
$$
if $c\not=0$, and otherwise
$$
\Delta_q(\gamma) (cz+d)^{-k}e(m\gamma\cdot z)=e(mz).
$$
\par
Moreover, we have obviously
$$
\Bigl|\Delta_q(\gamma) (cz+d)^{-k}e(m\gamma\cdot z)
\Bigr|\leq |cz+d|^{-k}
$$
and since $k\geq 4$, this defines an absolutely convergent series for
all $z$. We therefore obtain
$$
P_{m,q}(z)\ra e(mz)
$$
for any $z\in U$. Finally, the function
$$
z\mapsto \sum_{\gamma\in \Gamma_{\infty}\backslash \Gamma}{|cz+d|^{-k}}
$$
being integrable on $U$, we obtain the result after integrating.
\end{proof}

Here is a simple application to show that such qualitative statements
are not entirely content-free:

\begin{corollary}[``Strong approximation'' for $GL(2)$-cusp forms]
\label{cor-simple}
  Let $\Aa$ be the ad\`ele ring of $\Qq$. For each irreducible,
  cuspidal, automorphic representation $\pi$ of $GL(2,\Aa)$ and each
  prime $p$, let $\pi_p$ be the unitary, admissible representation of
  $GL(2,\Qq_p)$ that is the local component of $\pi$ at $p$. Then, for
  any finite set of primes $S$, as $\pi$ runs over the cuspidal
  spectrum of $GL(2,\Aa)$ unramified at primes in $S$, the set of
  tuples $(\pi_p)_{p\in S}$ is dense in the product over $p\in S$ of
  the unitary tempered unramified spectrum $X_p$ of $GL(2,\Qq_p)$.
\end{corollary}

\begin{proof}
  This is already known, due to Serre~\cite{serre} (if one uses
  holomorphic forms) or Sarnak~\cite{sarnak} (using Maass forms), but
  we want to point out that this is a straightforward consequence of
  Proposition~\ref{pr-k}; for more details, see the Appendix
  to~\cite{kst}. We first recall that the part of unitary unramified
  spectrum of $GL(2,\Qq_p)$ with trivial central character can be
  identified with $[-2\sqrt{p},2\sqrt{p}]$ via the map sending Satake
  parameters $(\alpha,\beta)$ to $\alpha+\beta$. The subset $X_p$ can
  then be identified with $[-2,2]$, and for $\pi=\pi(f)$ attached to a
  cuspidal primitive form unramified at $p$, the local component
  $\pi_p(f)$ corresponds to the normalized Hecke eigenvalue $\lambda_f(p)$.
\par
Now the (well-known) point is that for any integer of the form
$$
m=\prod_{p\in S}{p^{n(p)}}\geq 1,
$$
and any cusp form $f$ of weight $k$ with Fourier coefficients
$n^{(k-1)/2}\lambda_f(n)$, the characteristic property
$$
\langle f,P_{m,k}(\cdot)\rangle=\frac{\Gamma(k-1)}{(4\pi m)^{k-1}}
m^{(k-1)/2}\lambda_f(m)
$$
of Poincar\'e series (see, e.g.,~\cite[Lemma 14.3]{ant}) implies that
$$
p_{m,k}(1)=\sum_{f\in H_k}{\omega_f\lambda_f(m)} =\sum_{f\in
  H_k}{\omega_f\prod_{p\in S}{U_{n(p)}(\lambda_f(p))}},\quad
\omega_f=\frac{\Gamma(k-1)}{(4\pi)^{k-1}}\frac{1}{\|f\|^2},
$$
where $H_k$ is the Hecke basis of weight $k$ and level $1$, $U_n$
denotes Chebychev polynomials, and $\|f\|$ is the Petersson norm.
Because the linear combinations of Chebychev polynomials are dense in
$C([-2p^{1/2},2p^{1/2}])$ for any prime $p$, the fact that
$$
\lim_{k\ra +\infty}{p_{m,k}(1)}=\delta(m,1)=
\begin{cases}
1&\text{ if all $n(p)$ are zero,}\\
0&\text{ otherwise,}
\end{cases}
$$
(given by Proposition~\ref{pr-k}) shows, using the Weyl
equidistribution criterion, that $(\pi_p(f))_{f\in H_k}$, when counted
with weight $\omega_f$, becomes equidistributed as $k\ra +\infty$ with
respect to the product of Sato-Tate measures over $p\in S$. Since each
factor has support equal to $[-2,2]=X_p$, this implies trivially the
result.
\end{proof}

\section{Siegel modular forms}

We now proceed to generalize the previous result to Siegel cusp forms;
although some notation will be recycled, there should be no
confusion. For $g\geq 1$, let $\Hh_g$ denote the Siegel upper
half-space of genus $g$
$$
\Hh_g=\{z=x+iy\in M(g,\Cc)\,\mid\, \T{z}=z,\quad y
\text{ is positive definite}\},
$$
on which the group $\Gamma_g=\Sp(2g,\Zz)$ acts in the
usual way 
$$
\gamma\cdot z=(az+b)(cz+d)^{-1}
$$
(see, e.g.,~\cite[Ch. 1]{klingen} for such basic facts; we always write
$$
\gamma=\begin{pmatrix}
a&b\\
c& d
\end{pmatrix}
$$
for symplectic matrices, where the blocks are themselves $g\times g$
matrices). Let $A_g$ denote the set of symmetric, positive-definite
matrices in $M(g,\Zz)$ with integer entries on the main diagonal and
half-integer entries off it. Further, let
$$
\Gamma_{\infty} =
\Bigl\{\pm\begin{pmatrix}1&s\\0&1\end{pmatrix}: s \in M(g,\Zz),\ s=\T{s}
\Bigr\}.
$$
\par
For $k\geq 2$, even,\footnote{\ Forms of odd weight $k$ do exist if
  $g$ is even, but behave a little bit differently, and we restrict to
  $k$ even for simplicity.} and a matrix $s\in
A_g$, the Poincar\'e series $\mathcal{P}_{s,k}$ is defined by
$$
\mathcal{P}_{s,k}(z)=\sum_{\gamma\in \Gamma_{\infty}\backslash \Gamma_g}{
\det(cz+d)^{-k}e(\Tr(s\ (\gamma\cdot z)))
}
$$
for $z$ in $\Hh_g$. This series converges absolutely and uniformly on
compact sets of $\Hh_g$ for $k>2g$; indeed, as shown by
Maass~\cite[(32), Satz 1]{maass}), the series
$$
\mathcal{M}_{s,k}(z)=\sum_{\gamma\in \Gamma_{\infty}\backslash
  \Gamma_g}{ |\det(cz+d)|^{-k}\exp(-2\pi\Tr(s\ \Imag(\gamma\cdot z)))}
$$
which dominates it termwise converges absolutely and uniformly on
compact sets (see also~\cite[p. 90]{klingen}; note that, in contrast
with the case of $SL(2,\Zz)$, one can not ignore the exponential
factor here to have convergence).  The Poincar\'e series
$\mathcal{P}_{s,k}$ is then a Siegel cusp form of weight $k$ for
$\Gamma_g$. Therefore, it has a Fourier expansion
$$
\mathcal{P}_{s,k}(z)=\sum_{t\in A_g}{p_{s,k}(t)e(\Tr(tz))},
$$
which converges absolutely and uniformly on compact subsets of $\Hh_g$.

\begin{theorem}[Orthogonality for Siegel-Poincar\'e
  series]\label{th:siegelortho}
  With notation as above, for any fixed $s$, $t\in A_g$, we have
$$
\lim_{k\ra +\infty}{p_{s,k}(t)}=\delta'(s,t)\frac{|\Aut(s)|}{2},
$$
where the limit is over even weights $k$, $\delta'(s,t)$ is the
Kronecker delta for the $\GL(g,\Zz)$-equivalence classes of $s$ and
$t$, and where $\Aut(s)=O(s,\Zz)$ is the finite group of integral
points of the orthogonal group of the quadratic form defined by $s$.
\end{theorem}

This result suggests to define the Poincar\'e series with an
additional constant factor $2/|\Aut(s)|$, in which case this theorem
is exactly analogous to Proposition~\ref{pr-k}. And indeed, this is
how Maass defined them~\cite{maass}.

\begin{proof}
We adapt the previous argument, writing first
$$
p_{s,k}(t)=\int_{U_g}{\mathcal{P}_{s,k}(z)e(-\Tr(tz))dz}
$$
where $U_g=U_g(y_0)$ will be taken to be the (compact) set of matrices
$$
U_g(y_0)=\mathcal{U}_g+iy_0\mathrm{Id},
$$
for some real number $y_0>1$ to be selected later, where
$$
\mathcal{U}_g=\{x\in M(g,\Rr)\,\mid\, x\text{ symmetric and }
|x_{i,j}|\leq 1/2\text{ for all } 1\leq i,j\leq g\} \;
$$
the measure $dz$ is again Lebesgue measure.
\par
Before proceeding, we first recall that
\begin{equation}\label{eq-exp-neg}
|e(\Tr(s\gamma \cdot z))|\leq 1
\end{equation}
for all $s\in A_g$, $\gamma\in \Gamma_g$ and $z\in \Hh_g$. Indeed,
since $s$ is a real matrix, we have
$$
|e(\Tr(s\gamma \cdot z))|=\exp(-2\pi\Tr(s\Imag(\gamma \cdot z)))
$$
and the result follows from the fact that
$$
\Tr(sy)\geq 0
$$
for any $s\in A_g$ and $y$ positive definite. To see the latter, we
write $y=\T{q}q$ for some matrix $q$, and we then have
$$
sy=s\T{q}q=q^{-1}tq
$$
with $t=qs\T{q}$; then $t$ is still positive, while $\Tr(sy)=\Tr(t)$,
so $\Tr(sy)\geq 0$.
\par
We then have the following Lemma:\footnote{\ This statement is used to
  replace the inequality~(\ref{eq-j}), which has no obvious analogue
  when $g\geq 2$.}

\begin{lemma}\label{lemma-strict} 
For any integer $g\geq 1$, there exists a real number $y_0>1$,
  depending only on $g$, such that for any $\gamma\in\Gamma_g$ written
$$
\gamma=\begin{pmatrix}a & b\\
c&d
\end{pmatrix},\quad (a,b,c,d)\in M(g,\Zz),
$$
with $c\not=0$ and for all  $z\in U_g(y_0)$, we have
\begin{equation}\label{eq-strict}
|\det(cz+d)|>1 \end{equation}
whereas if $c=0$, we have $|\det(cz+d)|=1$.
\end{lemma}

Assuming the truth of this lemma, we find that
\begin{equation}\label{eq-dominated}
|\det(cz+d)^{-k}e(\Tr(s\gamma \cdot z))|\leq |\det(cz+d)|^{-2g-1}
\exp(-2\pi\Tr(s\Imag(\gamma \cdot z)))
\end{equation}
for any $k>2g$, all $z\in U_g$ and $\gamma\in\Gamma_{\infty}\backslash
\Gamma_g$, and also that
\begin{equation}\label{eq-siegellimit}
\det(cz+d)^{-k}e(\Tr(s\gamma \cdot z))\lra 0\text{ as } k\ra +\infty,
\end{equation}
for all $z\in U_g$ and all $\gamma$ with $c\not=0$. On the other hand,
if $c=0$, we have
$$
\gamma=\begin{pmatrix}
a& 0\\
0& \T{a^{-1}}
\end{pmatrix},
$$
up to $\Gamma_{\infty}$-equivalence, where $a\in \GL(g,\Zz)$ and hence
\begin{align*}
\det(cz+d)^{-k}e(\Tr(s\gamma \cdot z))&=
e(\Tr(s az \T{a})))=
e(\Tr(az\T{a}s))\\
&=
e(\Tr(\T{a}saz))
=e(\Tr((a\cdot s) z))
\end{align*}
where $a\cdot s=\T{a}sa$ (we use here that $k$ is even).
\par
Using~\eqref{eq-dominated},~\eqref{eq-siegellimit} and the absolute
convergence of $\mathcal{M}_{s,2g+1}(z)$, we find that
$$
\mathcal{P}_{s,k}(z)\lra \sum_{a\in \GL(g,\Zz)/\pm1}{e(\Tr((a\cdot s)z))}
$$
as $k\ra +\infty$, for all $z\in U_g$. (The series converges as a
subseries of the Poincar\'e series.)
\par
Then we multiply by $e(-\Tr(tz))$ and integrate over $U_g$,
using~(\ref{eq-dominated}) and the fact that $\mathcal{M}_{s,2g+1}$ is
bounded on $U_g$ to apply the dominated convergence theorem, and
obtain
$$
p_{s,k}(t)\lra \sum_{a\cdot s=t}{1},
$$
a number which is either $0$, if $s$ and $t$ are not equivalent, or
has the same cardinality as $\Aut(s)/2$ if they are. This completes
the proof of Theorem~\ref{th:siegelortho}.
\end{proof}

We still need to prove Lemma~\ref{lemma-strict}. We are going to use
the description of the Siegel fundamental domain $\mathcal{F}_g$ for
the action of $\Gamma_g$ on $\Hh_g$. Precisely, $\mathcal{F}_g$ is the
set of $z \in \Hh_g$ satisfying all the following conditions:
\begin{enumerate}
\item For all $\gamma\in\Gamma_g$, we have
$$
|\det(cz+d)|\geq 1;
$$
\item The imaginary part $\Imag(z)$ is Minkowski-reduced;
\item The absolute value of all coefficients of $\Reel(z)$ are $\leq
  1/2$.
\end{enumerate}
\par
Siegel showed that the first condition can be weakened to a finite
list of inequalities (see, e.g.,~\cite[Prop. 3.3, p. 33]{klingen}):
there exists a finite subset $C_g\subset \Gamma_g$, such that $z \in
\Hh_g$ belongs to $\mathcal{F}_g$ if and only if $z$ satisfies (2),
(3) and
\begin{equation}\label{eq-finite}
  |\det(cz+d)|\geq 1\text{ for all } \gamma\in C_g\text{ with } c\not=0.
\end{equation}
\par
Moreover, if~(\ref{eq-finite}) holds with equality sign for some
$\gamma\in C_g$, then $z$ is in the boundary of $\mathcal{F}_g$; if
this is not the case, then $|\det(cz+d)|>1$ for all $\gamma\in
\Gamma_g$ with $c\not=0$.

\begin{proof}[Proof of Lemma~\ref{lemma-strict}]
  First, we show that if $y_0>1$ is chosen large enough, the matrix
  $iy_0\mathrm{Id}$ is in $\mathcal{F}_g$. The only condition that
  must be checked is~(\ref{eq-finite}) when $\gamma\in C_g$ satisfies
  $c\not=0$, since the other two are immediate (once the definition of
  Minkowski-reduced is known; it holds for $y_0\mathrm{Id}$ when
  $y_0\geq 1$). For this, we use the following fact, due to
  Siegel~\cite[Lemma 9]{siegel} (see also~\cite[Lemma 3.3,
  p. 34]{klingen}): for any fixed $z=x+iy\in\Hh_g$ and any
  $\gamma\in\Gamma_g$ with $c\not=0$, the function
$$
\alpha\mapsto |\det(c(x+i\alpha)+d)|^2
$$
is strictly increasing on $[0,+\infty[$ and has limit $+\infty$ as
$\alpha\ra +\infty$. Taking $z=i$, we find that
$$
\lim_{y\ra+\infty}{|\det(iyc+d)|}=+\infty
$$
for every $\gamma\in C_g$. In particular, since $C_g$ is finite, there
exists $y_0>1$ such that
$$
|\det(cz_0+d)|>1
$$
for $z_0=iy_0$ and $\gamma\in C_g$, which is~(\ref{eq-finite}) for
$iy_0$.
\par
Because $\mathcal{U}_g$ is compact, it is now easy to also extend this
to $z=x+iy_0$ with $x\in \mathcal{U}_g$. Precisely, for fixed
$\gamma\in \Gamma_g$ with $c\not=0$, the function
$$
\begin{cases}
  \mathcal{U}_g\times ]0,+\infty[ \ra \Rr\\
  (x,\alpha) \mapsto |\det(c(x+i\alpha)+d)|^2
\end{cases}
$$
is a polynomial in the variables $(x,\alpha)$. As a polynomial in
$\alpha$, as observed by Siegel, it is in fact a polynomial in
$\alpha^2$ with non-negative coefficients, and it is non-constant
because $c\not=0$. (It is not difficult to check that the degree, as
polynomial in $\alpha$, is $2\rank(c)$). This explains the limit
$$
\lim_{y\ra +\infty}{|\det(c(x+iy)+d)|^2}=+\infty,
$$
but it shows also that it is uniform over the compact set
$\mathcal{U}_g$, and over the $\gamma\in C_g$ with
$c\not=0$. Therefore we can find $y_0$ large enough so
that~(\ref{eq-finite}) holds for all $z\in U_g$, and indeed holds with
the strict condition $|\det(cz+d)|>1$ on the right-hand side. By the
remark after~(\ref{eq-finite}), this means that $z$ is not in the
boundary of $\mathcal{F}_g$, and hence~(\ref{eq-strict}) holds for all
$\gamma$ with $c\not=0$.
\end{proof}

\begin{remark}
  The argument is very clear when $\det(c)\not=0$: we write
\begin{align*}
\det(c(x+iy)+d)&=\det(iyc)\det(1-iy^{-1}c^{-1}(cx+d))\\
&=
(iy)^g\det(c) (1+O(y^{-1}))
\end{align*}
for fixed $(c,d)$, uniformly for $x\in \mathcal{U}_g$.
\end{remark}

\begin{remark}
  It would be interesting to know the optimal value of $y_0$ in
  Lemma~\ref{lemma-strict}. For $g=1$, any $y_0>1$ is suitable. For
  $g=2$, Gottschling~\cite[Satz 1]{gott} has determined a finite set
  $C_2$ which determines as above the Siegel fundamental domain,
  consisting of $19$ pairs of matrices $(c,d)$; there are $4$ in which
  $c$ has rank $1$, $c$ is the identity for the others. Precisely: for
  $c$ of rank $1$, $(c,d)$ belongs to
$$
\Bigl\{
\Bigl(\begin{pmatrix}
1&0\\
0&0
\end{pmatrix},\begin{pmatrix}
0&0\\
0&1
\end{pmatrix}
\Bigr),
\Bigl(\begin{pmatrix}
0&0\\
0&1
\end{pmatrix},\begin{pmatrix}
1&0\\
0&0
\end{pmatrix}
\Bigr),
\Bigl(\begin{pmatrix}
1&-1\\
0&0
\end{pmatrix},\begin{pmatrix}
1&\text{$0$ or $1$}\\
-2&1
\end{pmatrix}
\Bigr)
\Bigr\},
$$
and for $c$ of rank $2$, we have $c=1$ and $d$ belongs to
$$
\Bigl\{
0,
\begin{pmatrix}
s&0\\
0& 0
\end{pmatrix},
\begin{pmatrix}
0&0\\
0& s
\end{pmatrix},
\begin{pmatrix}
s&0\\
0& s
\end{pmatrix},
\begin{pmatrix}
s&0\\
0& -s
\end{pmatrix},
\begin{pmatrix}
0&s\\
s& 0
\end{pmatrix},
\begin{pmatrix}
s&s\\
s& 0
\end{pmatrix},
\begin{pmatrix}
0&s\\
s& s
\end{pmatrix}
\Bigr\}
$$
where $s\in \{-1,1\}$. It should be possible to deduce a value of
$y_0$ using this information. Indeed, quick numerical experiments
suggest that, as in the case $g=1$, any $y_0>1$ would be suitable.
\end{remark}

\begin{remark}
  Analogues of Corollary~\ref{cor-simple} can not be derived
  immediately in the setting of Siegel modular forms because the link
  between Fourier coefficients and Satake parameters is much more
  involved; the case $g=2$ is considered, together with further
  applications and quantitative formulations, in~\cite{kst}.
\end{remark}


\begin{thebibliography}{CC}
\bibitem{gott}
E. Gottschling: \textit{Explizite Bestimmung der Randfl\"achen des
  Fundamentalbereiches der Modulgruppe zweiten Grades},
Math. Ann. 138 (1959), 103--124. 

\bibitem{ant} H. Iwaniec and E. Kowalski: \textit{Analytic number
    theory}, Coll. Publ. 53, A.M.S 2004.

\bibitem{klingen}
H. Klingen: \textit{Introductory lectures on Siegel modular forms},
Cambridge Studies Adv. Math. 20, Cambridge Univ. Press 1990.

\bibitem{kst}
E. Kowalski, A. Saha and J. Tsimerman: \textit{Local spectral
  equidistribution for Siegel modular forms and applications}, in
progress. 

\bibitem{maass} H. Maass: \textit{\"Uber die Darstellung der
    Modulformen $n$-ten Grades durch Poincar\'esche Reihen},
  Math. Annalen 123 (1951), 125--151.

\bibitem{serre}
J-P. Serre: \textit{R\'epartition asymptotique des valeurs propres de
  l'op\'erateur de Hecke $T_p$}, J. American Math. Soc. 10 (1997),
75--102. 

\bibitem{siegel}
C.L. Siegel: \textit{Symplectic geometry}, American J. Math. 65
(1943), 1--86; \url{www.jstor.org/stable/2371774}

\bibitem{sarnak}
P. Sarnak: \textit{Statistical properties of eigenvalues of the Hecke
  operator}, in ``Analytic Number Theory and Diophantine Problems'',
75--102, Progess in Math. 60, Birk\"auser, 1987.

\end{thebibliography}
\end{document}